\documentclass[11pt,reqno]{amsart}
\usepackage{amssymb}
\usepackage{amsthm,amsmath}
 \usepackage[colorlinks=true]{hyperref}
\hypersetup{urlcolor=blue, citecolor=red}
\usepackage{enumerate}

\newtheorem{theorem}{Theorem}%[section]
\newtheorem{proposition}{Proposition}

\newtheorem{lemma}{Lemma}
\theoremstyle{remark}
\newtheorem{example}{Example}
\newtheorem{remark}{Remark}

\newcommand{\B}{\mathcal{B}}
\newcommand{\Ex}{\mathbb{E}}
\newcommand{\eqd}{\stackrel{d}{=}}
\newcommand{\Exx}{\mathbb{E}_x}
\newcommand{\Px}{\mathbb{P}_x}
\newcommand{\Py}{\mathbb{P}_y}

\newcommand{\E}{\mathcal{E}}
\newcommand{\J}{\mathcal{P}}

\begin{document}

\title{Self-similar solutions of fragmentation equations
revisited}

\author{Weronika Biedrzycka}

\author{Marta Tyran-Kami\'nska}
%\author[Weronika Biedrzycka and Marta Tyran-Kami\'nska]{Weronika Biedrzycka and Marta Tyran-Kami\'nska}
\address{Institute of
Mathematics, University of Silesia, Bankowa 14, 40-007 Katowice, Poland}
\email{wsiwek@us.edu.pl}
\email{mtyran@us.edu.pl}


\thanks{This research was supported by the Polish NCN grant No. 2014/13/B/ST1/00224}
\date{\today}

\begin{abstract}
 We study the large time behaviour  of the mass (size) of particles described by the fragmentation equation  with homogeneous breakup kernel.  We give necessary and sufficient conditions for the convergence of solutions to the unique self-similar solution.
\end{abstract}

\maketitle

\section{Introduction}
Fragmentation is a phenomenon of breaking up particles into a range of smaller sized particles, characteristic of  many natural processes ranging from e.g. polymer degradation~\cite{ziff86} in chemistry to breakage of
aggregates \cite{arinorudnicki04} in biology. A stochastic model of fragmentation was given in \cite{filippov61} for the first time and since then it has been studied extensively with  probabilistic
methods, see  \cite{bertoin06,haas10,wagner10,wieczorek15} and the references therein. There is also a deterministic approach through transport equations and  purely functional-analytic methods   \cite{melzak57,lamb97,banasiaklamb03,arlottibanasiak04,escobedo05,banasiakarlotti06,mischler16}, which we follow here.
We denote by $c(t,x)$ the number density of particles  of mass  (size)  $x>0$ at time $t>0$.
The equation describing the evolution of density is %of particles in fragmentation is
\begin{equation}\label{eq:emass}
\dfrac {\partial c(t,x)}{\partial t}=\int_{x}^\infty
b(x,y)a(y)c(t,y)dy-a(x)c(t,x),\quad t, x>0,
\end{equation}
with initial condition
\begin{equation}\label{eq:emass0}
c(0,x)=c_0(x),\quad x>0,
\end{equation}
where $a(x)$ is the breakage rate for particles of mass $x$ and $b(x,y)$ is  the production rate of particles of size $x$ from those of size $y$.  Both $a$ and $b$ are nonnegative Borel measurable functions. To ensure the conservation of the total mass the function $b$ has to satisfy
\begin{equation*}
\int_0^yb(x,y)xdx=y %\quad \text{for all }y>0
\quad\text{and}\quad b(x,y)=0\quad \text{for} \quad x\ge y.
\end{equation*}
If we let $\gamma(y,x)=b(x,y)a(y)$ then \eqref{eq:emass} has the
same form as in \cite{melzak57}, where the reader can find a brief physical interpretation and the derivation of the equation. For a discussion of the model we also  refer to \cite[Chapter 8]{banasiakarlotti06}.

In this paper we provide necessary and sufficient conditions for the existence of self-similar
solutions to the initial value problem~\eqref{eq:emass}--\eqref{eq:emass0}
when $a(x)=x^\alpha$ with $\alpha>0$ and the kernel $b$ is \emph{homogenous}, i.e., there exists a Borel measurable function $h\colon
(0,1)\to\mathbb{R}_+$ such that
\begin{equation}\label{eq:hk}
b(x,y)=\frac{1}{y}h\left(\frac{x}{y}\right) \quad \text{for}\quad 0<x<y %,
\quad \text{and}\quad \int_{0}^1h(r)rdr=1.
\end{equation}
This means that the size $x$ of fragments of particles is proportional to the size $y$ of fragmenting particles, so that it is determined through the distribution of the ratio $x/y$ and does not depend on $y$.
We let $E$ denote the open interval $(0,\infty)$, $\E=\B(0,\infty)$ be the $\sigma$-algebra of Borel subsets of $(0,\infty)$, and $L^1=L^1(E, \E,m)$ be
the space of functions integrable with respect to the measure $m(dx)=x\,dx$ with the norm
\[
\|f\|=\int_0^\infty|f(x)|xdx.
\]
If $a(x)=x^\alpha$ with $\alpha\ge  0$ then the total mass is conserved \cite{lamb97,banasiak01}, so that if
\[
c_0\in D(m):=\{f\in L^1\colon f\ge 0, \|f\|=1\}
\]
then $c(t,\cdot)\in D(m)$ for all $t>0$, and if $\alpha<0$ the solutions loose mass due to the so called ``shattering'' phenomenon \cite{mcgradyziff87,haas03,banasiak04a,banasiakmokhtar05}.
For $\alpha>0$ we can represent \cite{michelmischler05,escobedo05} the solution $c$ of the fragmentation equation \eqref{eq:emass} as
\begin{equation*}%\label{eq:evd0}
c(t,x)=\gamma(t)^2 u(\log \gamma(t),\gamma(t)x),\quad x,t>0, \quad \text{where
}\gamma(t)=(1+t)^{1/\alpha}
\end{equation*}
and  $u$ is the solution of the
following partial integro-differential equation
\begin{equation}\label{eq:evd2}
\dfrac {\partial u(t,x) }{\partial t}=-\frac{1}{x}\frac{\partial}{\partial x}(x^2u(t,x))-\varphi(x) u(t,x)+
P(\varphi u(t,\cdot))(x).
\end{equation}
Here the function $\varphi$ and the linear operator $P$ on $L^1$ are defined as
\begin{equation}\label{eq:MoPf}
\varphi(x)=\alpha x^\alpha\quad \text{and}\quad  Pf(x)=\int_x^\infty
\frac{1}{y}h\Bigl(\frac{x}{y}\Bigr)f(y)dy,\quad  x>0.
\end{equation}
In particular, if $u_*$ is a stationary solution of \eqref{eq:evd2} then
 \begin{equation}\label{eq:sss}
c_*(t,x)=\gamma(t)^2u_*(\gamma(t)x),\quad t\ge 0,x>0,\quad \gamma(t)=(1+t)^{1/\alpha},
\end{equation}
is said to be a \emph{self-similar solution} of \eqref{eq:emass}. It should be noted that  $c$ behaves as  a delta-like function; see Remark~\ref{r:sw}.

Our main result is the
following.

\begin{theorem}\label{thm:asc}
Let $a(x)=x^\alpha$ with $\alpha >0$ and let $b$ be as in \eqref{eq:hk}.
There exists a self-similar solution $c_*$ as in \eqref{eq:sss} with  $c_*(t,\cdot)\in D(m)$, $t\ge0$,  and every solution $c$ of~\eqref{eq:emass} with initial condition
$c_0\in D(m)$ satisfies
\begin{equation*}
\lim_{t\to\infty} \int_0^\infty |c(t,x) - c_*(t,x)|x\,dx=0
\end{equation*}
if and only if
\begin{equation}\label{eq:flm}
\int_0^1 \log z \,h(z)z dz>-\infty.
\end{equation}
\end{theorem}

%It is interesting to note that if \eqref{eq:flm} holds then every solution $c$ of \eqref{eq:emass} with initial condition $c_0\in D(m)$ satisfies
%\[
%\lim_{t\to\infty} \int_{x_0}^\infty c(t,x)xdx=0\quad \text{for all }x_0>0,
%\]
%since
%%If \eqref{eq:flm} does not hold then this is a consequence of Proposition~\ref{p:sweeping}, since $\gamma(t)\ge 1$ and $c\ge 0$. While
%\[
%\int_{x_0}^\infty c(t,x)xdx\le \int_{0}^\infty |c(t,x)-c_*(t,x)|xdx+\int_{x_0}^\infty c_*(t,x)xdx
%\]
%and the right-hand side converges to zero as $t\to \infty$.

To our knowledge  this result is new in the given generality. Particular self-similar solutions were obtained in \cite{peterson86,ziff91} using various methods. For a probabilistic approach to self-similar fragmentation we refer the reader to \cite{bertoin06} and the references therein.
In the mathematical literature, the self-similar solutions  and the asymptotic behaviour for fragmentation equation has been considered using deterministic analytic methods in \cite{escobedo05} where the existence of self-similar solution is proved under the assumption that
\[
\int_0^1z^k h(z)dz<\infty
\]
for some $k<1$ and the convergence is proved when $k\le \alpha$ with additional regularity constraints. The proof of our result is based on the representation of solutions of \eqref{eq:emass} as  densities of a  Markov process $\{Y(t)\}_{t\ge 0}$; see \eqref{d:Yt} and~\eqref{eq:cM}.
We define  another Markov process  $\{X(t)\}_{t\ge 0}$ corresponding to the growth-fragmentation equation \eqref{eq:evd2} such that
\begin{equation}\label{eq:two}
Y(t)=(1+t)^{-1/\alpha}X(\log (1+t)^{1/\alpha}) \quad \text{ for all } t\ge 0.
\end{equation}
This is a  piecewise deterministic Markov process  with good asymptotic behaviour, as will be shown in Section~\ref{s:markov} using results from \cite{biedrzyckatyran} based on properties of stochastic semigroups.
In Lemma~\ref{l:Ex_t1} we give a relation between condition \eqref{eq:flm} and the first jump time of the process $\{X(t)\}_{t\ge 0}$. At the end of Section~\ref{s:markov}  we also define such processes in the case when $\alpha<0$.

Asymptotic properties of growth-fragmentation type equations have been studied and improved in many works, see e.g. \cite{mischler16} and \cite{bertoinwatson16}  for recent approaches, however none of these results allow us to provide the necessary and sufficient condition \eqref{eq:flm}.
The case where $\alpha=0$ is treated in \cite{doumicescobedo16} and  \cite{bertoinwatson16}.
If $\alpha>0$ then it is known \cite{bertoincaballero02} that the process $\{1/Y(t)\}_{t\ge 0}$
is an example of a so-called self-similar Markov processes with increasing sample paths
and it follows from \cite[Theorem 1]{bertoincaballero02} that  $t^{1/\alpha} Y(t)$ converges in distribution to a random variable $Y_\infty$, which is non-degenerate under condition  \eqref{eq:flm}. Our approach provides convergence of densities of $(1+t)^{1/\alpha} Y(t)$, implying convergence in distribution and identifies the distribution of $Y_\infty$ as being a stationary distribution of the process $\{X(t)\}_{t\ge0}$,  absolutely continuous with respect to $m$ and with density $u_*$.
If \eqref{eq:flm} does not hold then $Y_\infty$  is equal to~$0$. At the end of Section \ref{s:remarks} we describe how to get this type of limit using our approach. Further study of asymptotic behaviour in this case is provided in \cite{caballerorivero09}. Self-similar Markov processes were also used in \cite{haas10} to get the large time behaviour of solutions of the fragmentation equation when $\alpha<0$.
%Condition \eqref{eq:flm}

\section{Preliminaries}

Let $(E,\E,m)$ be a $\sigma$-finite
measure space and $L^1=L^1(E,\E,m)$ be the space of integrable functions. We denote by $D(m)\subset L^1$ the set
of all \emph{densities} on $E$, i.e.
\[
D(m)=\{f\in L^1_+: \|f\|=1\}, \quad \text{where } L^1_+=\{f\in L^1: f\ge 0\},
\]
and $\|\cdot\|$ is the norm in $L^1$.
A linear operator $P\colon L^1\to L^1$ such that $P(D(m))\subseteq D(m)$
is called \emph{stochastic} or \emph{Markov}~\cite{almcmbk94}. It is called \emph{substochastic} if $P$ is
a positive contraction, i.e., $Pf\ge 0$  and $\|Pf\|\le \|f\|$ for all $f\in  L_+^1$.

Let $\J \colon E\times \E\to[0,1]$ be a \emph{stochastic transition kernel}, i.e., $\J (x,\cdot)$ is a
probability measure for each $x\in E$ and the function $x\mapsto\J (x, B)$ is measurable for each $B\in\E$, and
let $P$ be a stochastic operator on $ L^1$. If
\begin{equation}\label{d:top}
\int_E \J (x,B)f(x)m(dx)=\int_B Pf(x)m(dx)
\end{equation}
for all $B\in \E, f\in D(m)$,
then $P$ is called the \emph{transition} operator corresponding to $\J $.
Suppose that there
exists a measurable function $p\colon E\times E\to[0,\infty)$ such that
\[
Pf(x) \ge \int_E p(x,y)f(y)\, m(dy)
\]
for $m$-a.e.~$x\in E$ and for every density $f$. If $p$ can be chosen in such a way that
\[
\int_E\int_E p(x,y)\,m(dx)\, m(dy) >0
\]
then  $P$ is called \emph{partially integral} and  if $p$ is such that
\[
\int_E p(x,y)\, m(dy) >0
\]
$m$-a.e.~$x\in E$ then  $P$ is called \emph{pre-Harris}.

We can extend a stochastic operator $P$ beyond the space $ L^1$  in the following way.
If $0\le f_n\le f_{n+1}$, $f_n\in L^1$, $n\in\mathbb{N}$, then the pointwise almost everywhere limit of $f_n$
exists and will be denoted by $\sup_n f_n$. For  $f\ge 0$  we define
\[
Pf=\sup_n Pf_n \quad \text{for }  f=\sup_n f_n, f_n\in L^1_+.
\]
(Note that $Pf$ is independent of the particular approximating sequence $f_n$ and that $Pf$ may be infinite.)
Moreover, if $P$ is the transition operator corresponding to $\J $ then \eqref{d:top} holds for all measurable
nonnegative $f$.  A~nonnegative measurable $f_*$ is said to be \emph{subinvariant
(invariant)} for a stochastic operator $P$ if
$Pf_*\le f_*$ ($Pf_*=f_*$).

Let $\E_*\subset \E$ be a given family of measurable subsets of $E$.
A stochastic operator $P$ is called \emph{sweeping} with respect to $\E_*$ if
\begin{equation*}%\label{sweeps1}
\lim_{n\to\infty}\int_{B} P^n f(x)m(dx)=0\quad\text{for all }f\in D(m), B\in \E_*.
\end{equation*}
A nonnegative measurable $f_*$ is said to be \emph{locally integrable}  with respect to $\E_*$  if
\[
\int_B f_*(x)m(dx)<\infty \quad \text{for all }B\in \E_*.
\]
\begin{proposition}[{\cite[Corollary 3]{rudnicki95}}]\label{p:sweep}
Suppose that a stochastic operator $P$ is pre-Harris and has no invariant density. If $P$  has a subinvariant $f_*$  with $f_*>0$ a.e. and  $f_*$  is locally integrable with respect to $ \E_*$,
 then the operator  $P$ is sweeping with respect to  $ \E_*$.
\end{proposition}

We conclude this section with the notion of stochastic semigroups and asymptotic behaviour of such semigroups.
A family of stochastic operators $\{P(t)\}_{t\ge 0}$ on $L^1$  which is a
$C_0$-\emph{semigroup}, i.e.,
\begin{enumerate}[\upshape (1)]
\item $P(0)=I$ (the identity operator);
\item $P(t+s)=P(t)P(s)$ for every $s,t\ge 0$;
\item for each $f\in L^1$ the mapping $t\mapsto P(t)f$ is continuous: for each $s\ge 0$
\[
\lim_{t\to s^{+}}\|P(t)f-P(s)f\|=0;
\]
\end{enumerate}
is called a \emph{stochastic semigroup}.
A nonnegative measurable $f_*$ is said to be
\emph{subinvariant (invariant)} for the semigroup $\{P(t)\}_{t\ge 0}$ if it is subinvariant (invariant) for each
operator $P(t)$.

 A
stochastic semigroup $\{P(t)\}_{t\ge 0}$ is called \emph{asymptotically stable} if it has an invariant
density $f_*$ such that
\begin{equation*}\label{d:as}
\lim_{t\to\infty}\|P(t) f-f_*\|=0\quad\text{for all }f\in D(m)
\end{equation*}
and  \emph{partially integral} if, for some $s>0$, the operator $P(s)$ is partially integral.
A stochastic semigroup  is called \emph{sweeping} with respect to  $\E_*$ if
\begin{equation*}\label{sweeps1}
\lim_{t\to\infty}\int_{B} P(t) f(x)m(dx)=0\quad\text{for all }f\in D(m), B\in \E_*.
\end{equation*}

\section{Construction of Markov processes}\label{s:markov}

In the first part of this section we construct two  Markov processes such that \eqref{eq:two} holds and their distributions are related to equations~\eqref{eq:emass} and ~\eqref{eq:evd2}. The second part contains the proof of Theorem~\ref{thm:asc}.  At the end of this section we discuss the case of $\alpha<0$. 

Let $\varepsilon_n, \theta_n$, $n\in \mathbb{N}$, be sequences of independent random variables, where the
$\varepsilon_n$ are exponentially distributed with mean $1$ and the $\theta_n$ are identically distributed with distribution function $H$  on $(0,1)$ of the form
\begin{equation}\label{d:H}
H(r)=\Pr(\theta_1\le r)=\int_0^r h(z)z dz,\quad r\in (0,1).
\end{equation}
If $Y_0$ is a positive random variable independent of $\theta_n$, $n\in \mathbb{N}$, then the sequence
\[
Y_{n}=\theta_n Y_{n-1},\quad n\ge 1,
\]
defines a discrete-time Markov process with stochastic transition kernel
\begin{equation}\label{eq:Jfr}
\mathcal{P}(x,B)=\int_{0}^1 1_B(z x)h(z)z dz, \quad x> 0, B\in \B(0,\infty).
\end{equation}
The transition operator $P$ on $L^1$ corresponding to $\mathcal{P}$ is as in \eqref{eq:MoPf}.

The process $\{Y(t)\}_{t\ge 0}$ is a pure jump Markov process \cite[Section 6.1]{tyran09} with the jump rate function $a(x)=x^\alpha$ and the  jump distribution $\mathcal{P}$ so that the process stays at $x$ for a random time, which is called a holding time and has an exponential distribution with mean $1/a(x)$, and then it jumps according to the probability distribution $\mathcal{P}(x,\cdot)$, independently on how long it stays at $x$.
Therefore we define the sample path of $\{Y(t)\}_{t\ge 0}$ starting at $Y(0)=Y_0=x$ as
\begin{equation}\label{d:Yt}
Y(t)=Y_n,\quad \tau_{n}\le t<\tau_{n+1}, n\ge 0,
\end{equation}
where $\tau_n$ are the jump times
\[
\tau_0=0,\quad \tau_n:=\sigma_n+\tau_{n-1},\quad
n\ge 1,
\]
and $\sigma_n$ are the holding times  defined by
\[
\sigma_n:=\frac{\varepsilon_n}{a(Y_{n-1})}=\frac{\varepsilon_n}{Y_{n-1}^\alpha},\quad n\ge 1.
\]
Note that (see e.g. \cite[Section 6.1]{tyran09}) if the probability density function of $Y(0)$ satisfies
\begin{equation}\label{eq:ipdf}
\Pr(Y(0)\in B)= \int_B c_0(x)x\,dx %,\quad \text{where}\quad c_0\in D(m)\cap L_\alpha^1,
\end{equation}
where $c_0\in D(m)$, then
\begin{equation}\label{eq:cM}
\Pr(Y(t)\in B)=\int_B c(t,x)x\,dx\quad \text{for all }t>0, B\in \B(0,\infty),
\end{equation}
where $c$ is the solution of equation~\eqref{eq:emass} with initial condition $c_0$.

The sample path of $\{X(t)\}_{t\ge 0}$ starting from $X(0)=X_0=x$ is defined as
\begin{equation*}
X(t)=e^{t-t_n}X_n,\quad t_{n}\le t< t_{n+1}, n\ge 0,
\end{equation*}
where   $t_n$ are the jump times
\begin{equation}\label{eq:t1}
t_0:=0,\quad  t_n=\log\left(\frac{\varepsilon_n}{X_{n-1}^\alpha}+1\right)^{1/\alpha}+t_{n-1}
\end{equation}
and $X_n=X(t_n)$ are the post-jump positions
\begin{equation}\label{d:xi}
X_{n}=\theta_n \left(\varepsilon_n+X_{n-1}^\alpha\right)^{1/\alpha},\quad n\ge 1.
\end{equation}
The process $\{X(t)\}_{t\ge 0}$, representing  fragmentation with growth,
 is the minimal piecewise deterministic Markov process \cite[Section 6.2]{tyran09} with characteristics $(\pi,\varphi,\mathcal{P})$,  where
\[
\varphi(x)=\alpha x^\alpha\quad\text{and}\quad \pi_tx=e^tx,\quad x>0,t\ge 0.
\]
This is a particular example of a semiflow with jumps as studied in \cite{biedrzyckatyran}, where the jumps are defined by the mappings $T_\theta(x)=\theta x$ and densities $p_\theta(x)=h(\theta)\theta$ for $x\in E=(0,\infty)$, $\theta\in \Theta=(0,1)$.

Now, if $X(0)=Y(0)$ and $Y(0)$ satisfies \eqref{eq:ipdf},  then
\[
\Pr(X(t)\in B)=\int_B P(t)c_0(x)xdx\quad  \text{for all }t>0, B\in \B(0,\infty),
\]
where $\{P(t)\}_{t\ge 0}$ is a stochastic semigroup on $L^1$ and $u(t,x)=P(t)c_0(x)$ is the solution of \eqref{eq:evd2} with initial condition $u(0,x)=c_0(x)$, see \cite[Section 6.2]{tyran09}.
This and~\eqref{eq:two} imply that the solution $c$ of equation \eqref{eq:emass} with initial condition $c_0$
can be represented as
\begin{equation}\label{eq:eq}
c(t,x)=\gamma(t)^2 P(\log \gamma(t))c_0(\gamma(t)x),\quad t>0, x>0, \text{ where
}\gamma(t)=(1+t)^{1/\alpha}.
\end{equation}
Hence, if the semigroup $\{P(t)\}_{t\ge 0}$ is asymptotically stable with invariant density $u_*$ then
\[
\lim_{t\to\infty}\int_{0}^\infty|c(t,x)-\gamma(t)^2u_*(\gamma(t)x)|x\,dx=\lim_{t\to\infty}\|P(\log
\gamma(t))c_0-u_*\|=0.
\]
Consequently, this reduces  the proof of Theorem~\ref{thm:asc} to the study of asymptotic stability of the
semigroup $\{P(t)\}_{t\ge 0}$.

For  the proof of Theorem~\ref{thm:asc} we need the following result \cite[Corollary 3.16]{biedrzyckatyran}, which is a refinement of \cite[Theorem 1.1]{biedrzyckatyran}.
\begin{theorem}\label{t:contf}
Assume that  the  semigroup  $\{P(t)\}_{t\geq 0}$ is partially integral and that the chain $(X(t_n))_{n\ge 0}$ defined in \eqref{d:xi} has only one invariant probability measure $\mu_*$,  absolutely continuous
with respect to~$m$. If the density $f_*=d\mu_{*}/dm$ is strictly positive  a.e., then $\{P(t)\}_{t\geq 0}$ is asymptotically stable if and only if
\begin{equation}\label{eq:R0v}
\Ex(t_1):=\int_{0}^\infty \Exx(t_1)f_*(x) m(dx) < \infty,
\end{equation}
where $t_1$ is the first jump time in \eqref{eq:t1} and $\Exx$ denotes the expectation operator with respect to the distribution $\Px$ of the process starting at $X(0) =x$.
\end{theorem}

We first show that all assumptions of Theorem~\ref{t:contf} are satisfied. We next prove in Lemma~\ref{l:Ex_t1} that conditions \eqref{eq:R0v} and \eqref{eq:flm} are equivalent.

\begin{lemma}\label{l:preH}
For each $t>0$ the operator $P(t)$ is pre-Harris. In particular, the semigroup $\{P(t)\}_{t\ge 0}$ is partially integral.
\end{lemma}
\begin{proof} Since the semigroup $\{P(t)\}_{t\ge 0}$ is stochastic, we have $m(y:\Py(t_\infty<\infty)>0)=0$ by \cite[Corollary 5.3]{tyran09}, where  $t_\infty=\lim_{n\to\infty}t_n$.
Observe that if $y$ is such that
$\Py(t_\infty<\infty)=0$, then
\[
\begin{split}
\Py(X(t)\in B)&=\sum_{n=0}^\infty \Py(X(t)\in B, t_n\le t< t_{n+1})
\end{split}
\]
and
$X(t)=e^{t-t_n}X(t_n)$ for $t\in [t_n,t_{n+1})$, $n\ge 0$.
For $n=1$ we have
\begin{multline*}
\Py(e^{t-t_1}X(t_1)\in B, t_1\leq t< t_{2})\\
=\int_{0}^1\int_0^t 1_B(e^{t }
\theta y)\psi_{t-s}(\theta e^sy)h(\theta)\theta \varphi(e^sy)\psi_s(y) ds d\theta,
\end{multline*}
where $\psi_t(y)=e^{-\int_0^t \varphi(e^ry)dr}$. The change of variables $x=e^t\theta y$ leads to
\[
\Py(e^{t-t_1}X(t_1)\in B, t_1\leq t< t_{2})=\int_Bp(x,y)xdx,
\]
where
\[
p(x,y)=1_{(0,e^ty)}(x)h\left(\frac{x}{e^ty}\right)\frac{1}{(e^ty)^2} \int_{0}^t \psi_{t-s }(e^{s-t}x)\varphi(e^sy)\psi_s(y)ds
\]
for $x,y>0$.
Hence
\[
\begin{split}
\int_B P(t)f(x)m(dx)&=\int_0^{\infty} \Py(X(t)\in B)f(y)m(dy)\\
&\ge \int_{0}^{\infty} \int_B p(x,y)m(dx) f(y)m(dy)
\end{split}
\]
for all $f\in D(m)$ and all Borel measurable sets $B$, which implies that
\[
 P(t)f(x)\ge \int_{0}^{\infty}p(x,y) f(y)m(dy),\quad f\in D(m).
\]
Observe that
\[
\int_0^\infty p(x,y)m(dy)>0\quad\text{for $m$-a.e. }x\in (0,\infty),
\]
which completes the proof.
\end{proof}

We will use the following lemma. Its proof is straightforward.

\begin{lemma}\label{l:positive_density}
Assume that $\xi$ and $\theta$ are independent random variables, where $\xi$ has a probability density function  $f_{\xi}$ on $(0,\infty)$, while $\theta$ has a density $f_{\theta}$ on $(0,1)$. Then the density $f_{\xi\theta}$ of the random variable $\xi\theta$ is given by
      \[
  f_{\xi\theta}(x)=\int_x^{\infty}f_{\theta}\left(\frac{x}{r}\right)\frac{1}{r}f_{\xi}(r) dr,\quad x>0,
  \]
      and it is positive a.e. if  $f_\xi$ is positive a.e.
\end{lemma}

Equality in distribution will be denoted by $\eqd$.

\begin{lemma}\label{l:X_infty} %Let $(\xi_n)_{n\ge 0}$ be
Let $\varepsilon_n, \theta_n$, $n\in \mathbb{N}$, be sequences of independent random variables, where the
$\varepsilon_n$ are exponentially distributed with mean $1$ and the $\theta_n$ are identically distributed with
distribution function $H$ as in \eqref{d:H}. Then the random variable
\begin{equation}\label{e:X_infty_def}
X_\infty=\left(\sum_{k\ge 1}\varepsilon_k\prod_{j=1}^{k}\theta_j^\alpha\right)^{1/\alpha}
\end{equation}
is finite a.e. and it satisfies
\begin{equation}\label{e:X_infty}
 X_\infty\eqd \theta_0 (\varepsilon_0 + X_\infty^\alpha)^{1/\alpha}
\end{equation}
with independent $X_\infty$, $\theta_0$, $\varepsilon_0$, where $\theta_0\eqd \theta_1$ and $
\varepsilon_0\eqd\varepsilon_1$.
Moreover,
\begin{equation}\label{e:EX_inf^alph}
\Ex X_{\infty}^{\alpha}=\frac{\Ex\theta_0^\alpha}{1-\Ex\theta_0^\alpha},
\end{equation}
the distribution $\mu_*$ of $X_{\infty}$ is absolutely continuous with respect to $m$ with strictly positive  density $f_*$,  and it is the unique stationary distribution of the Markov chain $(X_n)_{n\ge 0}$ defined
in~\eqref{d:xi}.
\end{lemma}

\begin{proof} It follows from \eqref{d:xi} that
\[
X_n^\alpha=\varepsilon_n\theta_n^\alpha+ X_{n-1}^\alpha \theta_n^\alpha,\quad n\ge 1.
\]
By iterating this equation we obtain
\[
X_n^\alpha=\varepsilon_n\theta_n^\alpha+\varepsilon_{n-1}\theta_{n-1}^\alpha\theta_n^\alpha+
\ldots+\varepsilon_1\prod_{j=1}^n\theta_j^\alpha+ X_{0}^\alpha\prod_{j=1}^n\theta_j^\alpha.
\]
Since $\theta_j^\alpha\in [0,1]$, $j\ge 1$, we see that the sequence $\prod_{j=1}^n\theta_j^\alpha$, being monotone, converges almost surely. In fact, it converges to zero, by the strong law of large numbers.
It is easily seen that
\[
X_n^\alpha-X_{0}^\alpha\prod_{j=1}^n\theta_j^\alpha\eqd \xi_n,\quad n\ge 1,
\]
where
\[
\xi_n=\varepsilon_1\theta_1^\alpha+\varepsilon_{2}\theta_{2}^\alpha\theta_1^\alpha+
\ldots+\varepsilon_n\prod_{j=1}^n\theta_j^\alpha,
\]
and the sequence $\xi_n$ converges almost surely to $X_\infty^\alpha$, where $X_\infty$ is as in \eqref{e:X_infty_def}. Therefore, if the Markov chain $(X_n)_{n\ge 0}$ has a stationary distribution then  it has to be the distribution of $X_\infty$.

Let the random variable $Z_{\infty}$ be defined by
\begin{equation}\label{d:Zinf}
Z_\infty=\sum_{k\ge 1}\varepsilon_k\prod_{j=1}^{k-1}\theta_j^\alpha.
\end{equation}
Note that in the right-hand side of~\eqref{d:Zinf}, we take the product to be equal to $1$ for $k =1$. Since the
random variables $\theta_j,\varepsilon_j$ are nonnegative, $Z_\infty$ is a well defined random variable with
values in $[0,\infty]$; in fact, it is finite almost surely \cite[Theorem 2.1]{goldie00} for our choice of $\theta_j,\varepsilon_j$.
Since  $-\infty\le \Ex(\log \theta_1)<0$ and $\Ex(\log (\max\{\varepsilon_1,1\})<\infty$, the series
in~\eqref{d:Zinf} converges almost surely, by \cite[Theorem 1.6]{vervaat79}, and
\begin{equation*}%\label{eq:Z}
Z_\infty\eqd \theta_0^\alpha Z_\infty
+\varepsilon_0,
\end{equation*}
where $\theta_0,\varepsilon_0,Z_\infty$ are independent with
$\theta_0\eqd\theta_1$, $\varepsilon_0\eqd \varepsilon_1$.
The Markov chain $(Z_n)_{n\ge 0}$ defined by
\[
Z_n=\theta_n^\alpha Z_{n-1}+\varepsilon_n,\quad n\ge 1,
\]
has a unique stationary distribution, by \cite[Theorem 1.5]{vervaat79}, which is the distribution of $Z_\infty$.
Now, observe that
\[
\theta_0^\alpha Z_\infty\eqd X_\infty^\alpha.
\]
Moreover, the random variable $ X_\infty^\alpha +\varepsilon_0$ has the same
distribution as the random variable $Z_\infty$. Consequently, the distribution of $X_\infty$ is the unique stationary distribution of the Markov chain $(X_n)_{n\ge 0}$.

Next observe that the distribution of $Z_\infty$, being a convolution of two distributions one of which  is absolutely continuous with respect to the Lebesgue measure, is  absolutely continuous. Hence, it has a probability density function $f_{Z_\infty}$.
Since $X_\infty\eqd \theta_0 Z_\infty^{1/\alpha}$ and the probability density function of $Z_\infty^{1/\alpha}$ is given by $\alpha x^{\alpha-1}f_{Z_\infty}(x^\alpha)$, the random variable $X_\infty$ also has a probability density function $f_{X_\infty}$, which implies that $f_*(x)x=f_{X_\infty}(x)$ for $x>0$.
To show that $f_*$ is positive a.s. it is enough to show, by Lemma \ref{l:positive_density}, that $f_{Z_\infty}$ is positive a.e.
Since $Z_\infty\eqd X_\infty^\alpha+\varepsilon_0$, we have
\[
f_{Z_\infty}(x)=\int_0^x e^{-(x-y)}f_{X_\infty^\alpha}(y)dy
\]
which shows that $f_{Z_\infty}$ is positive on an interval $(x_0,\infty)$, where $x_0\ge 0$.  To complete the proof, it remains to show that $x_0=0$. From  \eqref{d:Zinf} it follows that  $f_{Z_\infty}(x)$ satisfies the following equation
\[
f_{Z_\infty}(x)=\int_{0}^x\int_{z}^{\infty}g_\alpha\Bigl(\frac{z}{y}\Bigr)\frac{1}{y}f_{Z_\infty}(y)dy
e^{-(x-z)}dz,
\]
where $g_\alpha$ is the probability density function of the random variable $\theta_0^\alpha$. By changing the order of integration, we obtain
\[
f_{Z_\infty}(x)=\int_{0}^\infty\int_{0}^{\min\{x,y\}}g_\alpha\Bigl(\frac{z}{y}\Bigr)\frac{1}{y}
e^{-(x-z)}dz f_{Z_\infty}(y)dy.
\]
Suppose that $x_0>0$. Then for every
$x<x_0$ and  every $y>x_0$ we obtain
\[
\int_{0}^{x}g_\alpha\Bigl(\frac{z}{y}\Bigr)\frac{1}{y}e^{-(x-z)}dz=0
\]
This implies that for every $r<1$ we have $ \int_{0}^{r}g_\alpha(t)dt=0, $ which contradicts the fact that
$\int_{0}^{1} g_\alpha(t)dt=1$ and shows that $x_0=0$.
Consequently, the density $f_*$  is positive a.e.

Finally, observe that we have $0\le \Ex\theta_0^{\alpha}<1$. In fact, since $\theta_0^{\alpha}\in[0,1]$, we have $\Ex\theta_0^{\alpha}\in[0,1]$. Suppose that $\Ex\theta_0^{\alpha}=1$. Then $\theta_0^{\alpha}=1$ a.e., which implies that $\theta_0=1$ a.e. Hence $\theta_0$ has no density, which leads to a contradiction.
From \eqref{e:X_infty_def} we calculate
\begin{equation*}
  \Ex X_{\infty}^{\alpha}=\sum_{k\ge 1}\Ex\varepsilon_k\prod_{j=1}^{k}\Ex\theta_j^\alpha=\sum_{k\ge 1}\prod_{j=1}^{k}\Ex\theta_0^\alpha =\sum_{k\ge 1}\left(\Ex\theta_0^\alpha\right)^k = \frac{\Ex\theta_0^\alpha}{1-\Ex\theta_0^\alpha},
\end{equation*}
which gives \eqref{e:EX_inf^alph} and completes the proof.
\end{proof}

Lemmas~\ref{l:preH} and \ref{l:X_infty} imply that all assumptions of Theorem~\ref{t:contf} hold. Theorem~\ref{thm:asc} now follows by combining Theorem~\ref{t:contf} and the next lemma.

\begin{lemma}\label{l:Ex_t1}
The random variable $t_1$ in \eqref{eq:t1} satisfies \eqref{eq:R0v}
if and only if $\Ex\log\theta_1>-\infty$, in which case
\[
\Ex (t_1)=\Ex(-\log\theta_1)=-\int_0^1\log z\, h(z)zdz.
\]
\end{lemma}

\begin{proof} We have $t_1= \frac{1}{\alpha}\log\frac{\varepsilon_1+X_{0}^{\alpha}}{X_{0}^{\alpha}}$, where $\varepsilon_1$ and $X_0$ are independent.
This leads to
\[
\frac{1}{\alpha}\log\varepsilon_1-\log X_{0} \leq t_1 \leq \frac{1}{\alpha}(\varepsilon_1+X_{0}^{\alpha}-1)-\log X_{0}.
\]
To calculate the first moment of $t_1$ we  take  $X_0\eqd X_\infty$. % and $\varepsilon_1\eqd \varepsilon_0$.
Since $X_0^\alpha\eqd X_{\infty}^{\alpha}$ is integrable by \eqref{e:EX_inf^alph} and $\Ex\log\varepsilon_1=-\gamma$, where $\gamma$ is the Euler--Mascheroni constant,
we obtain that $t_1$ has a finite first moment if and only if $|\Ex\log X_{\infty}|<\infty$.
We have
\begin{equation*}
  \begin{split}
     \Ex(t_1)&=\int_{0}^{\infty}\int_0^{\infty} \log\left(\frac{y}{x^{\alpha}}+1\right)^{\frac{1}{\alpha}}e^{-y} dy f_*(x) m(dx) \\
       & = \int_0^1\int_0^{\infty}\int_0^{\infty}\left(\log\theta(y+x^{\alpha})^{\frac{1}{\alpha}}-\log x-\log\theta\right)e^{-y} dy f_*(x) m(dx) d\theta\\
       & =\Ex\log X_{\infty}-\Ex\log X_{\infty} -\Ex\log\theta_1 = -\Ex\log\theta_1.
   \end{split}
\end{equation*}
Moreover, from \eqref{e:X_infty} it follows that
\begin{equation*}
      \log X_{\infty}^{\alpha} \eqd  \alpha\log\theta_0+\log\varepsilon_0+\log\left(1+\frac{X_{\infty}^{\alpha}}{\varepsilon_0}\right)
         \geq \alpha\log\theta_0+\log\varepsilon_0.
\end{equation*}
On the other hand,
\begin{equation*}
     \log X_{\infty}^{\alpha} \eqd \alpha\log\theta_0+\log(\varepsilon_0+X_{\infty}^{\alpha})
        \leq \alpha\log\theta_0+\varepsilon_0+X_{\infty}^{\alpha}-1.
\end{equation*}
Hence
\[
\Ex\log\theta_0+\frac{1}{\alpha}\Ex\log\varepsilon_0\leq\Ex\log X_{\infty}\leq \Ex\log\theta_0+\frac{1}{\alpha}\Ex X_{\infty}^{\alpha},
\]
which implies that  $\Ex\log\theta_1>-\infty$ if and only if $|\Ex\log X_{\infty}|<\infty$ and completes the proof.
\end{proof}

We conclude this section with a construction
of the jump Markov process $\{Y(t)\}_{t\ge 0}$ corresponding to the fragmentation equation  \eqref{eq:emass} with $a(x)=x^\alpha$  and $\alpha<0$. The sample path of $Y(t)$ starting at $Y(0)=Y_0$ is defined as in  \eqref{d:Yt} as long as $t\in [\tau_n,\tau_{n+1})$ for some $n$. Since
\[
Y_0^\alpha \tau_n=Y_0^\alpha \sum_{k=1}^n\frac{\varepsilon_k}{Y_{k-1}^\alpha}= \sum_{k=1}^n \varepsilon_k \prod_{j=1}^{k-1}\theta_j^{-\alpha},\quad n\ge 1,
\]
and $-\alpha>0$, we see, as in the proof of Lemma~\ref{l:X_infty}, that the limit
\[
\lim_{n\to\infty}Y_0^\alpha \tau_n=\sum_{k=1}^\infty \varepsilon_k \prod_{j=1}^{k-1}\theta_j^{-\alpha}=:I_\infty
\]
exists and is finite a.s.  Thus the explosion time of the process, being defined by
$
\tau_\infty=\lim_{n\to\infty}\tau_n,
$
is finite a.s. The sequence $Y_n$, $n\ge 0$, is non-increasing and converging to $0$ a.s. Consequently, we can set $Y(t)=0$ for $t\ge \tau_\infty$ and say that $Y_0^{-\alpha}I_\infty$ is the first time when $Y(t)$ reaches $0$. Observe that  we have
\[
Y(t)=(1+t)^{-1/\alpha} X(\log(1+t)^{-1/\alpha}),\quad t\ge 0,
\]
where the corresponding piecewise deterministic Markov process $\{X(t)\}_{t\ge 0}$ has characteristics $(\pi,\varphi,\mathcal{P})$ with
\[
\varphi(x)=|\alpha|x^\alpha, \quad \pi_tx=e^{-t}x,\quad x>0,t\ge 0,
\]
and the evolution equation as in \cite[Section 6.3]{tyran09} or \cite[Section 4]{arinorudnicki04}.

Another representation of $\{Y(t)\}_{t\ge 0}$ is as follows. Let $\{Z(t)\}_{t\ge 0}$ be a compound Poisson process of the form
\[
Z(t)=-\sum_{j=1}^{N(t)}\log \theta_j,\quad t>0,
\]
where  $\{N(t)\}_{t\ge 0}$ is a Poisson process  with jump times
$\tilde{\tau}_n=\sum_{j=1}^n\varepsilon_j$, $n\ge 1$.
Then  $\{Y(t)\}_{t\ge 0}$ can be represented in the form \cite{haas10}
\[
Y(t)=Y_0e^{-Z(\rho(Y_0^\alpha t))},\quad t\ge 0,
\]
where $\rho$ is the time-change given by
\[
\rho(t)=\inf\{r\ge 0:\int_{0}^{r}e^{\alpha Z(s)}ds>t\},\quad t\ge 0.
\]
Note that
\[
\int_{0}^{\rho(t)}e^{\alpha Z(s)}ds=t
\]
if and only if $t<I_\infty$, and $\rho(t)=+\infty$ otherwise.
The random variable $I_\infty$ is an example of the so-called exponential functional (see e.g. \cite{carmonapetityor97})
\[
I_\infty=\int_0^\infty e^{\alpha Z(s)}ds.
\]
This can be easily seen by noting that $N(s)=k$ for $t\in[\tilde{\tau}_k,\tilde{\tau}_{k+1})$ with $\tilde{\tau}_0:=0$,
$\varepsilon_{k+1}=\tilde{\tau}_{k+1}-\tilde{\tau}_k$, $k\ge 0$, and
\[
\begin{split}
\int_{0}^{\infty}e^{\alpha Z(s)}ds&=\sum_{k\ge 0}\int_{\tilde{\tau}_k}^{\tilde{\tau}_{k+1}}\prod_{j=1}^{k}\theta_j^{-\alpha} ds
= \tilde{\tau}_1+
\sum_{k\ge 1} (\tilde{\tau}_{k+1}-\tilde{\tau}_k)\prod_{j=1}^{k}\theta_j^{-\alpha}=I_\infty.
%&=\varepsilon_1+
%\sum_{k\ge 1}\prod_{j=1}^{k}\theta_j^\alpha \varepsilon_{k+1}.
%\int_{0}^{\tilde{\tau}_1}dt +\int_{\tilde{\tau}_1}^{\infty}\prod_{j=1}^{N(t)} \theta_j^\alpha dt=\tilde{\tau}_1+\sum_{k\ge
%1}\int_{\tilde{\tau}_k}^{\tilde{\tau}_{k+1}}\prod_{j=1}^{k}\theta_j^\alpha dt
\end{split}
\]
To get finiteness of $I_\infty$ in terms of pathwise properties of the  process $\{Z(t)\}_{t\ge 0}$, one can simply assume that $\Ex(Z(1))\in (0,\infty)$, which is equivalent to \eqref{eq:flm}.

\section{Examples and final remarks}\label{s:remarks}
We have proved that the homogeneous fragmentation equation has a self-similar solution if and only if $\Ex(\log\theta_1)>-\infty$ where $\theta_1$ is a random variable with distribution function $H$ as in
\eqref{d:H}.  In that case the growth fragmentation equation has an integrable stationary solution.
Here we give the formula for the stationary solution in terms of the probability density function of the random variable  $Z_\infty$ as defined in \eqref{d:Zinf}.

Using the notation of \cite{biedrzyckatyran} note that the stochastic transition kernel $\mathcal{K}$ of
the Markov chain $(X_n)_{n\ge 0}$ defined in \eqref{d:xi}
is of the form
\[
\mathcal{K}(x,B)=\int_{0}^\infty \mathcal{P} (e^sx,B)\varphi(e^sx)e^{-\int_{0}^s\varphi(e^rx)dr}ds, \quad x>0, B\in\B(0,\infty),
\]
see e.g. \cite[Theorem 3.14]{biedrzyckatyran}, where $\varphi(x)=\alpha x^\alpha$ and $\mathcal{P}$ is the stochastic transition kernel defined in \eqref{eq:Jfr}.
We have
\begin{equation}\label{eq:KL}
\int_{0}^\infty \mathcal{K}(x,B)
f(x)m(dx)=\int_{0}^\infty\mathcal{P}(x,B)\varphi(x)R_0 f(x)m(dx)
\end{equation}
for all $B\in\B(0,\infty)$ and $f\in D(m)$, where
\begin{equation*}\label{d:Ro}
R_0 f(x)=\int_{0}^\infty e^{-2s} e^{(e^{-s}x)^\alpha -x^\alpha} f(e^{-s}x)ds,\quad x> 0, f\in D(m).
\end{equation*}
Since $P$ in \eqref{eq:MoPf} is the transition operator corresponding to $\mathcal{P}$, we obtain, by \eqref{eq:KL} and \eqref{d:top},
\[
\int_{0}^\infty\mathcal{K}(x,B)f(x)m(dx)=\int_{B} P(\varphi R_0 f)(x)m(dx),\quad B\in\B(0,\infty),  f\in D(m).
\]
Hence the stochastic operator $K$,  being the transition operator on $L^1$ corresponding to $\mathcal{K}$, is given by
\[
Kf=P(\varphi R_0 f), \quad f\in D(m).
\]
Note that $f_*$ in Lemma~\ref{l:X_infty} is the unique invariant density of the stochastic operator $K$, thus
\[
f_*=P(\varphi \overline{f}_*), \quad \text{where } \overline{f}_*=R_0 f_*.
\]
It follows from \cite[Theorem 3.3, Proposition 3.8]{biedrzyckatyran} that $\overline{f}_*$ is subinvariant  for the semigroup $\{P(t)\}_{t\ge 0}$.  We have $\overline{f}_*>0$ a.e., by Lemma~\ref{l:X_infty} and \cite[Corollary 3.4]{biedrzyckatyran} and the semigroup $\{P(t)\}_{t\ge 0}$ can have at most one invariant density, by~\cite[Theorem 3.15]{biedrzyckatyran}.
Note that $\overline{f}_*$ might not be integrable, but taking $B=(0,\infty)$ and $f=f_*$ in \eqref{eq:KL} shows that
\begin{equation}\label{e:integ}
\int_{0}^\infty\varphi(x)\overline{f}_*(x)m(dx)=\int_{0}^\infty f_*(x)m(dx)=1.
\end{equation}
From the proof of \cite[Theorem 3.15]{biedrzyckatyran} it follows that $\Ex(t_1)=\|\overline{f}_*\|$. Thus $\|\overline{f}_*\|=\Ex(-\log\theta_1)$, by Lemma~\ref{l:Ex_t1}.

On the other hand, $f_*(x)x$ is the probability density function of the random variable $X_\infty\eqd \theta_0 Z_\infty^{1/\alpha}$ with $\theta_0\eqd \theta_1$ and the operator $P$ corresponds to a multiplication by $\theta_1$.
Thus the probability density function of $Z_\infty^{1/\alpha}$ satisfies
\begin{equation*}%\label{eq:sts}
\alpha x^{\alpha-1}f_{Z_\infty}(x^\alpha)=\varphi(x) \overline{f}_*(x)x.
\end{equation*}
Consequently, if $\Ex(-\log\theta_1)<\infty$
then  the semigroup $\{P(t)\}_{t\ge 0}$ has a unique invariant density  $u_*$ and it is given by
\begin{equation*}\label{e:ustar}
u_*(x)=\frac{\overline{f}_*(x)}{\|\overline{f}_*\|}=\frac{f_{Z_\infty}(x^\alpha)}{\Ex(-\log\theta_1)x^2}.
\end{equation*}
We have proved the following.
\begin{proposition}\label{p:scaling} If \eqref{eq:flm} holds, equivalently $\Ex(-\log\theta_1)<\infty$,  then the self-similar solution of equation \eqref{eq:emass}
in Theorem~\ref{thm:asc}  is of the form
\[
c_*(t,x)=\frac{f_{Z_\infty}((1+t)x^\alpha)}{\Ex(-\log\theta_1)x^2},\quad x,t>0,
\]
where $f_{Z_\infty}$ is the probability density function of the random variable  $Z_\infty$ in~\eqref{d:Zinf}.
\end{proposition}
This form of the self-similar solution should be compared with the scaling assumption  in \cite[equation (15)]{ziff91}, where the function $\Phi$ corresponds to $f_{Z_\infty}$.

We now give an exactly solvable example, known since the \cite{peterson86}.
\begin{example}\label{ex:g} Consider the function
$h(z)=\beta z^{\beta-2}$ with $\beta>0$. Then $\theta_1\eqd U^{1/\beta}$, where $U$ is a random variable with uniform
distribution on $[0,1]$. Since $Z_\infty$ has the gamma distribution with shape parameter $1+\beta/\alpha$ (see e.g.~\cite[Example 3.8]{vervaat79}), we have
\[
f_{Z_\infty}(x)=\frac{1}{\Gamma(1+\beta/\alpha)}x^{\beta/\alpha}e^{-x}, \quad x>0,
\]
where $\Gamma$ is the gamma function. This implies that
\[
u_*(x)=\frac{\beta}{\Gamma(1+\beta/\alpha)}x^{\beta-2}e^{-x^\alpha}
=\frac{\alpha}{\Gamma(\beta/\alpha)}x^{\beta-2}e^{-x^\alpha}.
\]
Note that $u_*(x)x$ is the probability density function of the generalized gamma distribution with parameters $(\alpha,\beta,1)$.
\end{example}

The random variable $Z_\infty$ can be represented as the exponential functional
\[
Z_\infty\eqd \int_{0}^{\infty}e^{-\alpha Z(t)}dt
\]
of the compound Poisson process $\{Z(t)\}_{t\ge 0}$ defined above.
The Laplace exponent $\phi(q)$ of $\{Z(t)\}_{t\ge 0}$, which is defined by
\[
\Ex(e^{-q Z(t)})=e^{-t\phi(q)},\quad t>0,
\]
is of the form $$\phi(q)=\Ex(1-\theta_1^{ q})=\int_0^1(1-z^{ q})zh(z)dz,\quad q>0.$$
Hence, \cite[Proposition 3.3]{carmonapetityor97} implies that the random variable $Z_\infty$ is determined by its moments
\begin{equation*}\label{e:moments}
\Ex(Z_\infty^n)=\frac{n!}{\prod_{k=1}^{n}\phi(\alpha k)}=\frac{n!}{\prod_{k=1}^{n}\Ex(1-\theta_1^{\alpha k})},\quad n=1,2,\ldots.
\end{equation*}

We next give two examples where the random variable $Z_\infty$ can be identified through its moments.

\begin{example}
Recall that a random variable $\theta$ has a beta distribution with parameters $(a,b)$, $a,b>0$, if its probability density function is
\[
f_\theta(x)=\frac{\Gamma(a+b)}{\Gamma(a)\Gamma(b)}x^{a-1}(1-x)^{b-1},\quad x\in (0,1).
\]
If $\theta_1$ is a product of two independent random variables  with beta distributions with parameters $(\beta_1,1)$ and $(\beta_2,1)$, then $\Ex(-\log \theta_1)=1/\beta_1+1/\beta_2$ and
$Z_\infty$ is a product of two independent random variables, one is beta distributed with parameters $(1+a_1,a_2)$ and the other has a gamma distribution with shape parameter $1+a_2$,
where
\begin{equation}\label{d:a}
a_1=\frac{\beta_1}{\alpha}, \quad a_2=\frac{\beta_2}{\alpha}.
\end{equation}
\end{example}

\begin{remark}\label{r:dens}
It is easily seen that if $Z_\infty=\theta\xi$ where $\theta$ and $\xi$ are independent random  variables,  $\theta$ has a beta distribution with parameters $(1+a_1,a)$ and $\xi$ has a gamma distribution with shape parameter $1+a_2$, then the probability density function of $Z_\infty$ is of the form
\[
f_{Z_\infty}(x)=\frac{\Gamma(1+a_1+a)}{\Gamma(1+a_1)\Gamma(1+a_2)}e^{-x}x^{a_2}U(a,1+a_2-a_1,x),\quad x>0,
\]
where $U$ is the confluent hypergeometric function of the second kind
\[
U(a,b,x)=\frac{1}{\Gamma(a)}\int_{0}^\infty e^{-xs}s^{a-1}(1+s)^{b-a-1}ds, \quad a,x>0,
b\in\mathbb{R}.
\]
\end{remark}

\begin{example}
As in \cite{ziff91} consider now the function
\[h(z)=p\beta_1 z^{\beta_1-2}+(1-p)\beta_2 z^{\beta_2-2},\quad z\in (0,1), \] where $\beta_1,\beta_2>0$, $p\in [0,1]$. We can assume that $\beta_2>\beta_1$ and $p>0$. Then $p(a_2-a_1)>0$, where $a_1,a_2$ are as in \eqref{d:a},  $\Ex(-\log\theta_1)=p/\beta_1+(1-p)/\beta_2$,
and
$Z_\infty$ is the product of two independent random variables, one is beta distributed with parameters $(1+a_1,p(a_2-a_1))$ and the other has a gamma distribution with parameter $1+a_2$. Remark~\ref{r:dens} and Proposition~\ref{p:scaling} allow us to recover the scaling solutions from \cite{ziff91}.
\end{example}

We conclude the paper by commenting on the behaviour of  solutions of the fragmentation equation when $\Ex(\log \theta_1)=-\infty$.
\begin{proposition}\label{p:sweeping}
If \eqref{eq:flm} does not hold then, for every solution $c$ of~\eqref{eq:emass} with initial condition
$c_0\in D(m)$, we have
\begin{equation*}
\lim_{t\to\infty}\int_{x_0/\gamma(t)}^\infty c(t,x)x\,dx=0\quad \text{for all }x_0>0.
\end{equation*}
\end{proposition}
\begin{proof}
If $\Ex(-\log\theta_1)=\infty$ then $\overline{f}_*$ is not integrable and the semigroup $\{P(t)\}_{t\ge 0}$ has no invariant density, since if there were one, then it would be a scalar multiple of $\overline{f}_*$, by \cite[Corollaries 3.11, 3.12]{biedrzyckatyran}.
Hence, for every $s>0$ the operator $P(s)$ does not have an invariant density, by \cite[Proposition 7.12.1]{almcmbk94}.
From Lemma~\ref{l:preH} and Proposition~\ref{p:sweep} it follows that the operator $P(s)$ is sweeping which together with  \cite[Theorem 7.11.1]{almcmbk94} implies that the semigroup $\{P(t)\}_{t\ge 0}$ is sweeping from every set $B$ satisfying
\[
\int_B \overline{f}_*(x)m(dx)<\infty.
\]
We have $\varphi(x)=\alpha x^\alpha \ge \varphi(x_0)>0$ for all $x\ge x_0>0$. From this and \eqref{e:integ}  we see that $\overline{f}_*$ is integrable over intervals $B=(x_0,\infty)$, $x_0>0$.  Consequently,
\[
\lim_{t\to\infty}\int_{x_0}^\infty P(t)c_0(x)xdx=0,
\]
which together with \eqref{eq:eq} completes the proof.
\end{proof}

\begin{remark}\label{r:sw} It is interesting to note that every solution $c$ of \eqref{eq:emass} with initial condition $c_0\in D(m)$ satisfies
\[
\lim_{t\to\infty} \int_{x_0}^\infty c(t,x)xdx=0\quad \text{for all }x_0>0.
\]
To see this observe that if \eqref{eq:flm} does not hold then this is a consequence of Proposition~\ref{p:sweeping}, since $\gamma(t)\ge 1$ and $c\ge 0$, while if \eqref{eq:flm} holds then
\[
\int_{x_0}^\infty c(t,x)xdx\le \int_{0}^\infty |c(t,x)-c_*(t,x)|xdx+\int_{x_0}^\infty c_*(t,x)xdx
\]
and the right-hand side converges to zero as $t\to \infty$, by Theorem~\ref{thm:asc} and integrability of $u_*$.
%\[
%\int_{x_0}^\infty c_*(t,x)xdx=\int_{\gamma(t)x_0}u_*(x)xdx
%\]
\end{remark}

\section*{Acknowledgments} We are indebted to the anonymous referees for their
comments that have materially improved this paper.

\end{document}